\documentclass[12pt]{amsart}

\usepackage{amsmath,amssymb,amsthm,here}
\usepackage[all]{xypic}
\usepackage[dvipdfm,colorlinks=true]{hyperref}

\numberwithin{equation}{section}

\SelectTips{eu}{12}

\newtheorem{theorem}{Theorem}[section]
\newtheorem{proposition}[theorem]{Proposition}
\newtheorem{lemma}[theorem]{Lemma}
\newtheorem{corollary}[theorem]{Corollary}

\theoremstyle{definition}

\newtheorem{conjecture}[theorem]{Conjecture}

\theoremstyle{remark}
\newtheorem{remark}[theorem]{Remark}

\newcommand{\Z}{\mathbb{Z}}

\renewcommand{\P}{\mathcal{P}}
\renewcommand{\t}{\mathtt{t}}

\title{Samelson products in $p$-regular exceptional Lie groups}

\author{Sho Hasui}
\address{Department of Mathematics, Kyoto University, Kyoto, 606-8502, Japan}
\email{s.hasui@math.kyoto-u.ac.jp}
\author{Daisuke Kishimoto}
\address{Department of Mathematics, Kyoto University, Kyoto, 606-8502, Japan}
\email{kishi@math.kyoto-u.ac.jp}
\author{Akihiro Ohsita}
\address{Faculty of Economics, Osaka University of Economics, Osaka 533-8533, Japan}
\email{ohsita@osaka-ue.ac.jp}

\begin{document}

\baselineskip 16pt

\maketitle

\begin{abstract}
The (non)triviality of Samelson products of the inclusions of the spheres into $p$-regular exceptional Lie groups is completely determined, where a connected Lie group is called $p$-regular if it has the $p$-local homotopy type of a product of spheres. 
\end{abstract}

\section{Introduction and statement of the result}

For a homotopy associative H-space with inverse $X$, the correspondence $X\wedge X\to X$, $(x,y)\mapsto xyx^{-1}y^{-1}$ induces a binary operation
$$\langle-,-\rangle:\pi_i(X)\otimes\pi_j(X)\to\pi_{i+j}(X)$$ 
called the Samelson product in $X$. We consider the basic Samelson products in $p$-regular Lie groups. Let $G$ be a compact simply connected Lie group. By the Hopf theorem, $G$ has the rational homotopy type of the product $S^{2n_1-1}\times\cdots\times S^{2n_\ell-1}$, where $n_1\le\cdots\le n_\ell$. The sequence $n_1,\ldots,n_\ell$ is called the type of $G$ and is denoted by $\t(G)$. We here list the types of exceptional Lie groups.
\renewcommand{\arraystretch}{1.2}
\begin{table}[H]
\centering
\begin{tabular}{l|l||l|l}
\hline
$G$&$\t(G)$&$G$&$\t(G)$\\\hline
$\mathrm{G}_2$&$2,6$&$\mathrm{E}_6$&$2,5,6,8,9,12$\\
$\mathrm{F}_4$&$2,6,8,12$&$\mathrm{E}_7$&$2,6,8,10,12,14,18$\\
&&$\mathrm{E}_8$&$2,8,12,14,18,20,24,30$\\\hline
\end{tabular}
\end{table}
\noindent We say that $G$ is $p$-regular if it has the $p$-local homotopy type of a product of spheres. By the classical result of Serre, it is known that $G$ is $p$-regular if and only if $p\ge n_\ell$, in which case 
$$G_{(p)}\simeq S^{2n_1-1}_{(p)}\times\cdots\times S^{2n_\ell-1}_{(p)}.$$
Suppose that $G$ is $p$-regular, and let $\epsilon_{2n_i-1}$ be the composite
$$S^{2n_i-1}\xrightarrow{\rm incl}S^{2n_1-1}_{(p)}\times\cdots\times S^{2n_\ell-1}_{(p)}\simeq G_{(p)}$$
where if there are more than one $i$ in $\t(G)$, we distinguish corresponding $\epsilon_{2i-1}$ but not write it explicitly. The Samelson products $\langle\epsilon_{2i-1},\epsilon_{2j-1}\rangle$ are fundamental in studying the homotopy (non)commutativity of $G_{(p)}$ as in \cite{KK} and its applications (See \cite{KKTh,KKTs,Th}, for example). So we would like to determine their (non)triviality. In \cite{B}, Bott computes the Samelson products in the classical groups $\mathrm{U}(n)$ and $\mathrm{Sp}(n)$. Then by combining with the information of the $p$-primary component of the homotopy groups of spheres \cite{To}, the (non)triviality of the Samelson products $\langle\epsilon_{2i-1},\epsilon_{2j-1}\rangle$ is completely determined when $G=\mathrm{SU}(n),\mathrm{Sp}(n),\mathrm{Spin}(2n+1)$, where $\mathrm{Sp}(n)_{(p)}\simeq\mathrm{Spin}(2n+1)_{(p)}$ as loop spaces by \cite{F} since $p$ is odd. For example, when $G=\mathrm{SU}(n)$ and $p\ge n$, the type of $G$ is given by $2,\ldots,n$ and
$$\langle\epsilon_{2i-1},\epsilon_{2j-1}\rangle\ne 0\quad\text{if and only if}\quad i+j>p.$$
So apart from $\mathrm{Spin}(2n)$, all we have to consider  is the exceptional Lie groups. The (non)triviality of the  Samelson products $\langle\epsilon_{2i-1},\epsilon_{2j-1}\rangle$ is known only in a few cases, and the most general result so far is: 

\begin{theorem}
[Hamanaka and Kono \cite{HK}]
\label{HK}
Let $G$ be a $p$-regular exceptional Lie group. If $i,j\in\t(G)$ satisfy $i+j=p+1$, then $\langle\epsilon_{2i-1},\epsilon_{2j-1}\rangle$ is nontrivial.
\end{theorem}

\begin{remark}
The Samelson products in $\mathrm{G}_2$ are first computed in \cite{O}, and some more Samelson products in $\mathrm{E}_7$ and $\mathrm{E}_8$ are computed in \cite{KK}.
\end{remark}

Based on this result, Kono posed the following conjecture (in a private communication).

\begin{conjecture}
\label{conj}
Let $G$ be a $p$-regular exceptional Lie group. For $i,j\in\t(G)$, there exists $k\in\t(G)$ satisfying $i+j=k+p-1$ if and only if $\langle\epsilon_{2i-1},\epsilon_{2j-1}\rangle$ is nontrivial.
\end{conjecture}

Notice that the only if part of the conjecture follows immediately from the information of the $p$-primary component of the homotopy groups of spheres \cite{To} (cf. \cite{KK}). We will prove the if part and obtain:

\begin{theorem}
\label{main}
Conjecture \ref{conj} is true.
\end{theorem}

The paper is structured as follows. In \S 2, we reduce the nontriviality of the Samelson products $\langle\epsilon_{2i-1},\epsilon_{2j-1}\rangle$ in the $p$-regular Lie group $G$ to a certain condition of the Steenrod operation $\P^1$ on the mod $p$ cohomology of the classifying space $BG$. Then for a $p$-regular exceptional Lie group $G$, we compute the mod $p$ cohomology of $BG$ as the ring of invariants of the Weyl group of $G$. With this description of the mod $p$ cohomology of $BG$, we compute the action of $\P^1$ on it. In \S 3, we prove that the above condition on $\P^1$ is satisfied to complete the proof of Theorem \ref{main}. 

\section{Mod $p$ cohomology of $BG$}

\subsection{Reduction}

Let $G$ be a compact simply connected Lie group. We first reduce Theorem \ref{main} to the action of the Steenrod operation $\P^1$ on the mod $p$ cohomology of the classifying space $BG$ as in \cite{HK,KK}. Recall that if the integral homology of $G$ has no $p$-torsion, the mod $p$ cohomology of the classifying space $BG$ is given by
\begin{equation}
\label{BG}
H^*(BG;\Z/p)=\Z/p[x_{2i}\,\vert\,i\in\t(G)],\quad|x_j|=j.
\end{equation}
When there are more than one $i$ in $\t(G)$, we distinguish corresponding $x_{2i}$ but not write it explicitly as in the case of $\epsilon_{2i-1}$ in the preceding section.

\begin{lemma}
\label{criterion}
Suppose that $G$ is $p$-regular. For $i,j\in\t(G)$, if there is $k\in\t(G)$ such that $\P^1x_{2k}$ involves $\lambda x_{2i}x_{2j}$ with $\lambda\ne 0$, then $\langle\epsilon_{2i-1},\epsilon_{2j-1}\rangle$ is nontrivial.
\end{lemma}

\begin{proof}
Let $\bar{\epsilon}_{2i}:S^{2i}\to BG_{(p)}$ be the adjoint of $\epsilon_{2i-1}$ for $i\in\t(G)$, and so we may assume that $\bar{\epsilon}_{2i}^*(x_{2i})=u_{2i}$ for a generator $u_{2i}$ of $H^{2i}(S^{2i};\Z/p)$. Assume that the Samelson product $\langle\epsilon_{2i-1},\epsilon_{2j-1}\rangle$ is trivial, which is equivalent to the triviality of the Whitehead product $[\bar{\epsilon}_{2i},\bar{\epsilon}_{2j}]$ by the adjointness of Samelson products and Whitehead products. Then the map $\bar{\epsilon}_{2i}\vee\bar{\epsilon}_{2j}:S^{2i}\vee S^{2j}\to BG_{(p)}$ extends to a map $\mu:S^{2i}\times S^{2j}\to BG_{(p)}$, up to homotopy. Hence since $\P^1x_{2k}$ involves $\lambda x_{2i}x_{2j}$ with $\lambda\ne 0$, we have
$$\mu^*(\P^1x_{2k})=\mu^*(\lambda x_{2i}x_{2j})=\lambda u_{2i}\times u_{2j}\ne 0.$$
On the other hand, by the naturality of $\P^1$, we also have
$$\mu^*(\P^1x_{2k})=\P^1\mu^*(x_{2k})=0$$
since $\P^1$ is trivial on $H^*(S^{2i}\times S^{2j};\Z/p)$, which is a contradiction. Therefore the proof is completed.
\end{proof}

By Lemma \ref{criterion}, we obtain the if part of Theorem \ref{main} by the following.

\begin{theorem}
\label{P}
Let $G$ be a $p$-regular exceptional Lie group. If $i,j,k\in\t(G)$ satisfy $i+j=k+p-1$, $\P^1x_{2k}$ involves $\lambda x_{2i}x_{2j}$ with $\lambda\ne 0$.
\end{theorem}

The rest of this paper is devoted to prove Theorem \ref{P}. 

\subsection{Generators}

In this subsection, we choose generators of the mod $p$ cohomology of $BG$. We set notation. Hereafter, let $p$ be a prime greater than $5$. Recall that the integral homology of $G$ is $p$-torsion free for $p>5$, and so the mod $p$ cohomology of $BG$ is given as \eqref{BG}. For a homomorphism $\rho:H\to K$ between Lie groups, we denote the induced map $BH\to BK$ ambiguously by $\rho$. 

We first choose generators of the mod $p$ cohomology of $B\mathrm{E}_8$. Let $T$ be a maximal torus of $\mathrm{E}_8$. Then as in \cite{MT}, since $p>5$, the inclusion $T\to \mathrm{E}_8$ induces an isomorphism

\begin{equation}
\label{W(E)}
H^*(B\mathrm{E}_8;\Z/p)\xrightarrow{\cong}H^*(BT;\Z/p)^{W(\mathrm{E}_8)},
\end{equation}
where the right hand side is the ring of invariants of the Weyl group $W(\mathrm{E}_8)$. We calculate invariants of $W(\mathrm{E}_8)$ through a maximal rank subgroup of $\mathrm{E}_8$. Let $\epsilon_1,\ldots,\epsilon_8$ be the standard basis of $\mathbb{R}^8$ which is regarded as the Lie algebra of $T$. As in \cite{MT}, we choose simple roots of $\mathrm{E}_8$ as
$$\alpha_1=\frac{1}{2}(\epsilon_1+\epsilon_8)-\frac{1}{2}(\epsilon_2+\epsilon_3+\epsilon_4+\epsilon_5+\epsilon_6+\epsilon_7),\quad\alpha_2=\epsilon_1+\epsilon_2,\quad\alpha_i=\epsilon_{i-1}-\epsilon_{i-2}\quad(3\le i\le 8),$$
by which the extended Dynkin diagram of $\mathrm{E}_8$ is described as
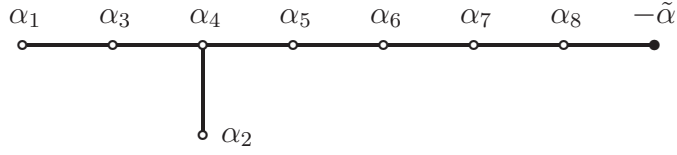
\begin{center}
\setlength{\unitlength}{1.2cm}
\begin{picture}(7.5,1.7)
\thicklines
\multiput(0.5,1.1)(1,0){7}{\circle{0.1}}
\put(7.5,1.1){\circle*{0.124}}
\put(2.5,0.1){\circle{0.1}}
\multiput(0.55,1.1)(1,0){7}{\line(1,0){0.9}}
\put(2.5,1.05){\line(0,-1){0.9}}
\put(0.35,1.35){$\alpha_1$}
\put(1.35,1.35){$\alpha_3$}
\put(2.35,1.35){$\alpha_4$}
\put(3.35,1.35){$\alpha_5$}
\put(4.35,1.35){$\alpha_6$}
\put(5.35,1.35){$\alpha_7$}
\put(6.35,1.35){$\alpha_8$}
\put(7.25,1.35){$-\tilde{\alpha}$}
\put(2.7,0.03){$\alpha_2$}
\end{picture}
\end{center}
where $\tilde{\alpha}$ is the dominant root. Removing $\alpha_1$ from the diagram, we get the maximal rank subgroup of $\mathrm{E}_8$ which is of type $\mathrm{D}_8$. Then there is a homomorphism $\rho_1:\mathrm{Spin}(16)\to \mathrm{E}_8$ which induces a monomorphism
$$\rho_1^*:H^*(B\mathrm{E}_8;\Z/p)\to H^*(B\mathrm{Spin}(16);\Z/p).$$
By putting $t_1=-\epsilon_1$, $t_8=-\epsilon_8$ and $t_i=\epsilon_i$ ($2\le i\le 7$), $H^*(BT;\Z/p)$ is identified with the polynomial ring $\Z/p[t_1,\ldots,t_8]$. Let $c_i$ and $p_i$ be the $i$-th elementary symmetric functions in $t_1,\ldots,t_8$ and in $t_1^2,\ldots,t_8^2$, respectively. As in \eqref{W(E)}, we have an isomorphism
$$H^*(B\mathrm{Spin}(16);\Z/p)\xrightarrow{\cong}\Z[t_1,\ldots,t_8]^{W(\mathrm{D}_8)}=\Z/p[p_1,\ldots,p_7,c_8],$$
and then since $W(\mathrm{E}_8)$ is generated by $W(\mathrm{D}_8)$ and the reflection $\varphi$ corresponding to the simple root $\alpha_1$, it follows from \eqref{W(E)} that 
$$H^*(B\mathrm{E}_8;\Z/p)\cong\Z/p[p_1,\ldots,p_7,c_8]\cap\Z/p[t_1,\ldots,t_8]^\varphi.$$ 
In \cite{HK}, the action of $\varphi$ on $p_1,\ldots,p_8,c_8\in\Z/p[t_1,\ldots,t_8]$ is described as 
$$\varphi(p_1)=p_1,\quad\varphi(p_i)\equiv p_i+h_ic_1,\quad\varphi(c_8)\equiv c_8-\frac{1}{4}c_7c_1\mod(c_1^2)$$
for $2\le i\le 8$, where
\begin{alignat*}{3}
h_2&=\frac{3}{2}c_3,&h_3&=-\frac{1}{2}(5c_5+c_3c_2),&h_4&=\frac{1}{2}(7c_7+3c_5c_2-c_4c_3),\\
h_5&=-\frac{1}{2}(5c_7c_2-3c_6c_3+c_5c_4),\quad&h_6&=-\frac{1}{2}(5c_8c_3-3c_7c_4+c_6c_5),\quad&h_7&=\frac{1}{2}(3c_8c_5-c_7c_6).
\end{alignat*}
We put
\begin{align*}
\hat{x}_4=&p_1,\\
\hat{x}_{16}=&12 p_4-\frac{18}{5}p_3p_1+p_2^2+\frac{1}{10}p_2p_1^2+168c_8,\\
\hat{x}_{24}=&60p_6-5p_5p_1-5p_4p_2+3p_3^2-p_3p_2p_1+\frac{5}{36}p_2^3+110c_8p_2,\\
\hat{x}_{28}=&480p_7+40p_5p_2-12p_4p_3-p_3p_2^2-3p_4p_2p_1+\frac{24}{5}p_3^2p_1+\frac{11}{36}p_2^3p_1+312c_8p_3-82c_8p_2p_1,
\end{align*}

\begin{align*}
\hat{x}_{36}=&480p_7p_2+72p_6p_3-30p_5p_4-\frac{25}{2}p_5p_2^2+9p_4p_3p_2-\frac{18}{5}p_3^3-\frac{1}{4}p_3p_2^3+1020c_8p_5+102c_8p_3p_2\\
&-42p_6p_2p_1+9p_5p_3p_1-\frac{3}{2}p_4p_2^2p_1+\frac{9}{5}p_3^2p_2p_1+\frac{1}{24}p_2^4p_1-330c_8p_4p_1-\frac{89}{2}c_8p_2^2p_1-300c_8^2p_1\\
&+\frac{89}{4}p_5p_2p_1^2-\frac{15}{2}p_4p_3p_1^2-\frac{11}{20}p_3p_2^2p_1^2+156c_8p_3p_1^2+\frac{5}{16}p_4p_2p_1^3+\frac{9}{8}p_3^2p_1^3+\frac{27}{320}p_2^3p_1^3\\
&-\frac{323}{8}c_8p_2p_1^3-\frac{195}{32}p_5p_1^4-\frac{13}{64}p_3p_2p_1^4-\frac{7}{192}p_2^2p_1^5+\frac{195}{32}c_8p_1^5+\frac{3}{32}p_3p_1^6-\frac{1}{1024}p_2p_1^7,\\
\hat{x}_{40}=&480p_7p_3+50p_6p_2^2+50p_5^2-10p_5p_3p_2-\frac{25}{2}p_4^2p_2+9p_4p_3^2-\frac{25}{36}p_4p_2^3+\frac{3}{4}p_3^2p_2^2+\frac{25}{864}p_2^5\\
&+2400c_8p_6+250c_8p_4p_2+3550c_8^2p_2+6c_8p_3^2-\frac{175}{18}c_8p_2^3,\\
\hat{x}_{48}=&-200p_7p_5-60p_7p_3p_2+3p_6p_3^2+\frac{25}{9}p_6p_2^3+\frac{25}{3}p_5^2p_2-\frac{5}{2}p_5p_4p_3-\frac{25}{24}p_5p_3p_2^2-\frac{25}{48}p_4^2p_2^2\\
&+p_4p_3^2p_2+\frac{25}{864}p_4p_2^4-\frac{3}{10}p_3^4-\frac{1}{36}p_3^2p_2^3-\frac{25}{62208}p_2^6-400c_8p_6p_2-115c_8p_5p_3-\frac{25}{12}c_8p_4p_2^2\\
&+3c_8p_3^2p_2+\frac{25}{27}c_8p_2^4+75c_8p_4^2-300c_8^2p_4-\frac{1525}{12}c_8^2p_2^2+300c_8^3.
\end{align*}
Hamanaka and Kono \cite{HK} calculates $\varphi$-invariants in dimension 4, 16 and 24 as follows.

\begin{proposition}
[Hamanaka and Kono \cite{HK}]
\label{4-24}
Let $\bar{x}_i\in\Z/p[p_1,\ldots,p_7,c_8]$ with $|\bar{x}_i|=i$. 
\begin{enumerate}
\item If $\varphi(\bar{x}_i)\equiv \bar{x}_i\mod(c_1^2)$ in $\Z/p[t_1,\ldots,t_8]$ for $i=4,16$, then  
$$\bar{x}_4=\alpha\hat{x}_4\quad\text{and}\quad\bar{x}_{16}=\beta\hat{x}_{16}+\gamma\hat{x}_4^4\quad(\alpha,\beta,\gamma\in\Z/p).$$
\item If $\varphi(\bar{x}_{24})\equiv \bar{x}_{24}\mod(c_1^2,c_2^2)$ in $\Z/p[t_1,\ldots,t_8]$, then
$$\bar{x}_{24}\equiv\alpha\hat{x}_{24}\quad(\alpha\in\Z/p).$$
\end{enumerate}
\end{proposition}

We further calculate $\varphi$-invariants in dimension $28,36,40,48$, where a partial calculation in dimension $28$ is given in \cite{KK}.

\begin{proposition}
[cf. \cite{KK}]
\label{28-48}
Let $\bar{x}_i\in\Z/p[p_1,\ldots,p_7,c_8]$ with $|\bar{x}_i|=i$.
\begin{enumerate}
\item If  $\varphi(\bar{x}_{28})\equiv \bar{x}_{28}\mod(c_1^2,c_2^2)$ in $\Z/p[t_1,\ldots,t_8]$, then
$$\bar{x}_{28}\equiv\alpha\hat{x}_{28}+\beta\hat{x}_4\hat{x}_{24}\mod(p_1^2)\quad(\alpha,\beta\in\Z/p).$$
\item If  $\varphi(\bar{x}_{36})\equiv \bar{x}_{36}\mod(c_1^2)$ in $\Z/p[t_1,\ldots,t_8]$, then
$$\bar{x}_{36}=\alpha_1\hat{x}_{36}+\alpha_2\hat{x}_4\hat{x}_{16}^2+\alpha_3\hat{x}_4^2\hat{x}_{28}+\alpha_4\hat{x}_4^3\hat{x}_{24}+\alpha_5\hat{x}_4^5\hat{x}_{16}+\alpha_6\hat{x}_4^9\quad(\alpha_i\in\Z/p).$$
\item If $\varphi(\bar{x}_i)\equiv \bar{x}_i\mod(c_1^2,c_2)$ in $\Z/p[t_1,\ldots,t_8]$ for $i=40,48$, then
$$\bar{x}_{40}\equiv\alpha_1\hat{x}_{40}+\alpha_2\hat{x}_{24}\hat{x}_{16},\quad\bar{x}_{48}\equiv\beta_1\hat{x}_{48}+\beta_2\hat{x}_{24}^2+\beta_3\hat{x}_{16}^3\mod(p_1)\quad(\alpha_i,\beta_i\in\Z/p).$$
\end{enumerate}
\end{proposition}

\begin{proof}
The proof is the same as Proposition \ref{4-24} given in \cite{HK}, and we only consider $\bar{x}_{28}$ since other cases are analogous. Excluding the indeterminacy $\hat{x}_4\hat{x}_{24}$, we may suppose that $\bar{x}_{28}$ is a linear combination
$$\lambda_1p_7+\lambda_2p_5p_2+\lambda_3p_4p_3+\lambda_4p_4p_2p_1+\lambda_5p_3^2p_1+\lambda_6p_3p_2^2+\lambda_7p_2^3p_1+\lambda_8c_8p_3+\lambda_9c_8p_2p_1$$
for $\lambda_i\in\Z/p$. By the congruence $\varphi(\bar{x}_{28})\equiv \bar{x}_{28}\mod(c_1^2,c_2^2)$ and the equality $p_i=\sum_{j+k=2i}(-1)^{i+j}c_jc_k$, we get linear equations in $\lambda_1,\ldots,\lambda_9$. Solving these equations, we see that $\bar{x}_{28}\equiv\alpha \hat{x}_{28}\mod(c_1^2,c_2^2)$, thus the proof is completed since the intersection of the ideal $(c_1^2,c_2^2)$ and the subring $\Z/p[p_1,\ldots,p_7,c_8]$ of $\Z/p[t_1,\ldots,t_8]$ is the ideal $(p_1^2)$ in $\Z/p[p_1,\ldots,p_7,c_8]$.
\end{proof}
As an immediate consequence of Proposition \ref{4-24} and \ref{28-48}, we obtain:

\begin{corollary}
\label{E8}
We can choose a generator $x_i$ of $H^*(B\mathrm{E}_8;\Z/p)$ for $i\ne 60$ in such a way that 
\begin{alignat*}{5}
&\rho_1^*(x_i)=\hat{x}_i&&(i=4,16,36),\quad&\rho_1^*(x_i)\equiv\hat{x}_i\mod(p_1^2)&\quad(i=24,28)\\
&\rho_1^*(x_i)\equiv\hat{x}_i\mod(p_1)&&\quad(i=40,48).
\end{alignat*}
\end{corollary}

Hereafter, we choose generators of $H^*(B\mathrm{E}_8,\Z/p)$ as in Corollary \ref{E8}. From these generators, we next choose generators of $H^*(BG;\Z/p)$ for $G=\mathrm{F}_4,\mathrm{E}_6,\mathrm{E}_7$. Recall that there is a commutative diagram of the canonical homomorphisms
\begin{equation}
\label{naturality}
\xymatrix{\mathrm{F}_4\ar[r]^{\alpha_3}&\mathrm{E}_6\ar[r]^{\alpha_2}&\mathrm{E}_7\ar[r]^{\alpha_1}&\mathrm{E}_8\\
\mathrm{Spin}(9)\ar[r]^{\theta_3}\ar[u]_{\rho_4}&\mathrm{Spin}(10)\ar[r]^{\theta_2}\ar[u]_{\rho_3}&\mathrm{Spin}(12)\ar[r]^{\theta_1}\ar[u]_{\rho_2}&\mathrm{Spin}(16)\ar[u]_{\rho_1}.}
\end{equation}
Let us consider the induced map of the arrows in the mod $p$ cohomology of the classifying spaces. Obviously, we have
\begin{alignat}{4}
\label{theta1}
\theta_1^*(p_i)&=p_i\;(i=1,2,3,4,5),\quad&\theta_1^*(p_6)&=c_6^2,\quad&\theta_1^*(p_7)&=0,\quad&\theta_1^*(c_8)&=0,\\
\label{theta2}
\theta_2^*(p_i)&=p_i\;(i=1,2,3,4),\quad&\theta_2^*(p_5)&=c_5^2,\quad&\theta_2^*(c_6)&=0,\\
\label{theta3}
\theta_3^*(p_i)&=p_i\;(i=1,2,3,4),&\theta_3^*(c_5)&=0.
\end{alignat}
To determine the induced map of $\alpha_i$, we recall the results of \cite{A,C,N,TW,W}.

\begin{proposition}
\label{G/H}
\begin{enumerate}
\item $H^*(\mathrm{E}_6/\mathrm{Spin}(10);\Z/p)=\Z/p[y_8]/(y_8^3)\otimes\Lambda(y_{17})$, $|y_i|=i$.
\item $H^*(\mathrm{E}_6/\mathrm{F}_4;\Z/p)=\Lambda(z_9,z_{17})$, $|z_i|=i$.
\item $\widetilde{H}^*(\mathrm{E}_7/\mathrm{E}_6;\Z/p)=\Z/p\langle z_{10},z_{18}\rangle$, $|z_i|=i$ for $*<37$.
\item $H^*(\mathrm{E}_8/\mathrm{E}_7;\Z/p)=\Z/p[z_{12},z_{20}]$, $|z_i|=i$ for $*<40$, .
\end{enumerate}
\end{proposition}

We next choose generators of $H^*(BG;\Z/p)$ for $G\ne\mathrm{E}_8$. Let
$$\hat{x}_{10}=c_5,\quad\hat{x}_{12}=-6p_3+p_2p_1-60c_6,\quad\hat{x}_{18}=p_2c_5\quad\text{and}\quad \hat{x}_{20}=p_5+p_2c_6.$$
We abbreviate $\theta_i(\hat{x}_j)$ by $\hat{x}_j$.

\begin{corollary}
\label{E7}
We can choose a generator $x_i$ of $H^*(B\mathrm{E}_7;\Z/p)$ as 
$$\rho_2^*(x_i)=\hat{x}_i\quad(i=4,12,16,36)\quad\text{and}\quad\rho_2^*(x_i)\equiv\hat{x}_i\mod(p_1^2)\quad(i=20,24,28).$$
\end{corollary}

\begin{proof}
Consider the Serre spectral sequence of the homotopy fiber sequence $\mathrm{E}_8/\mathrm{E}_7\to B\mathrm{E}_7\to B\mathrm{E}_8$. Then by Proposition \ref{G/H}, we get $\alpha_1^*(x_i)=x_i$ for $i=4,16,24,28,36$, hence the desired result for $\rho_2^*(x_i)$ by Corollary \ref{E8}. As in \cite{BH}, we can choose a generator $x_{12}$ of $H^*(B\mathrm{F}_4;\Z/p)$ as $\rho_4^*(x_{12})=-6p_3+p_2p_1$. On the other hand, it is calculated in \cite{N} that $\rho_2^*(x_{12})\equiv-6p_3-60c_6$ modulo decomposables. Then we get $\rho_2^*(x_{12})=\hat{x}_{12}$ by \eqref{theta2} and \eqref{theta3}. By the Serre spectral sequence of the homotopy fiber sequence $\mathrm{E}_6/\mathrm{Spin}(10)\to B\mathrm{Spin}(10)\to B\mathrm{E}_6$ and Proposition \ref{G/H}, we have $\rho_3^*(x_{10})\ne 0$. Then for a degree reason, we may choose $x_{10}\in H^*(B\mathrm{E}_6;\Z/p)$ as $\rho_3^*(x_{10})=c_5$. Consider next the Serre spectral sequence of the homotopy fiber sequence $\mathrm{E}_7/\mathrm{E}_6\to B\mathrm{E}_6\to B\mathrm{E}_7$. Then it follows from Proposition \ref{G/H} that we may choose $x_{20}\in H^*(B\mathrm{E}_7;\Z/p)$ as $\alpha_2^*(x_{20})=x_{10}^2$, hence $\rho_2^*(x_{20})\equiv p_5+\alpha p_2c_6\mod(p_1^2)$ by \eqref{theta2}, where $\alpha\in\Z/p$. For a degree reason, we have $\alpha_1^*(x_{40})\equiv\lambda x_{20}^2\mod(x_4,x_{12},x_{16})$, hence
$$\theta_2^*(\hat{x}_{40})=\lambda(p_5+\alpha p_2c_6)^2\mod(\hat{x}_4,\hat{x}_{12},\hat{x}_{16}).$$
Since $\theta_2^*(\hat{x}_{40})\equiv 50p_5^2-10p_5p_3p_2+\frac{1}{2}p_3^2p_2^2$ and $\hat{x}_{20}^2\equiv p_5^2-\frac{\alpha}{5}p_5p_3p_2+\frac{\alpha^2}{100}p_3^2p_2^2\mod (\hat{x}_4, \hat{x}_{12}, \hat{x}_{16})$, we get $\alpha=1$ and $\lambda=50$
\end{proof}

\begin{corollary}
\label{E6}
We can choose a generator $x_i$ of $H^*(B\mathrm{E}_6;\Z/p)$ as 
$$\rho_3^*(x_i)=\hat{x}_i\quad(i=4,12,10,16,18)\quad\text{and}\quad\rho_3^*(x_{24})=\hat{x}_{24}\mod(p_1^2).$$
\end{corollary}

\begin{proof}
By the  Serre spectral sequence of the homotopy fiber sequence $\mathrm{E}_7/\mathrm{E}_6\to B\mathrm{E}_6\to B\mathrm{E}_7$ together with Proposition \ref{G/H} and Corollary \ref{E7}, we get $\alpha_2^*(x_i)=x_i$ for $i=4,12,16,24$. Then we obtain the desired result for $x_i$ ($i=4,12,16,24$) by Corollary \ref{E7}. The second equality holds for a degree reason as mentioned in the proof of Corollary \ref{E7}. We see that $\rho_3^*(x_{28})$ involves the term $p_2c_5^2$, implying $\rho_3^*(x_{18})\ne 0$. Then for a degree reason, we may choose $x_{18}$ as desired.
\end{proof}

\begin{corollary}
\label{F4}
We can choose a generator $x_i$ of $H^*(B\mathrm{F}_4;\Z/p)$ as
$$\rho_4^*(x_i)=\hat{x}_i\quad(i=4,12,16)\quad\text{and}\quad\rho_4^*(x_{24})\equiv\hat{x}_{24}\mod(p_1^2).$$
\end{corollary}

\begin{proof}
The result follows from the Serre spectral sequence of the homotopy fiber sequence $\mathrm{E}_6/\mathrm{F}_4\to B\mathrm{F}_4\to B\mathrm{E}_6$ together with Proposition \ref{G/H} and Corollary \ref{E6}.
\end{proof}

Recall that $\mathrm{G}_2$ is a subgroup of $\mathrm{Spin}(7)$. We denote the inclusion $\mathrm{G}_2\to\mathrm{Spin}(7)$ by $\rho$.

\begin{proposition}
\label{G2}
The induced map of $\rho:B\mathrm{G}_2\to B\mathrm{Spin}(7)$ in the mod $p$ cohomology satisfies
$$\rho^*(p_1)=x_4,\quad\rho^*(p_2)=0\quad\text{and}\quad\rho^*(p_3)=x_{12}.$$
\end{proposition}

\begin{proof}
It is well known that $\mathrm{Spin}(7)/\mathrm{G}_2=S^7$. Then by considering the Serre spectral sequence of the homotopy fiber sequence $\mathrm{Spin}(7)/\mathrm{G}_2\to B\mathrm{G}_2\to B\mathrm{Spin}(7)$, we obtain the desired result.
\end{proof}

For the rest of this paper, we choose generators of $H^*(BG;\Z/p)$ as in Corollary \ref{E7}, \ref{E6}, \ref{F4}, \ref{G2}. 

\subsection{Calculation of $\P^1\rho_i^*(x_j)$}

Let $s_k$ denote the power sum $t_1^{2k}+\cdots+t_m^{2k}$. The power sum $s_k$ can be expanded in terms of $p_1,\ldots,p_m$ by the Girard's formula
\begin{equation}
\label{Girard}
s_k=(-1)^kk\sum_{i_1+2i_2+\cdots+mi_m=k}(-1)^{i_1+\cdots+i_m}\frac{(i_1+\cdots+i_m-1)!}{i_1!\cdots i_m!}p_1^{i_1}\cdots p_m^{i_m}.
\end{equation}
By the following lemma together with Girard's formula, we can determine $\P^1p_n$ and $\P^1c_m$.

\begin{lemma}
\label{P1}
For $q=\frac{p-1}{2}$, there holds
$$\P^1p_n=2\sum_{i=0}^{q-1}(-1)^ip_{n+i}s_{q-i}+2(-1)^q(n+q)p_{n+q}\quad\text{and}\quad\P^1c_m=c_ms_q.$$
\end{lemma}

\begin{proof}
For $1\le i_1<\cdots<i_k\le m$, we have
\begin{align*}
(t_{i_1}^{2m}+\cdots+t_{i_k}^{2m})t_{i_1}^2\cdots t_{i_k}^2
=&s_mt_{i_1}^2\cdots t_{i_k}^2-s_{m-1}\sum_{i_0\ne i_1,\ldots,i_k}t_{i_0}^2t_{i_1}^2\cdots t_{i_k}^2\\
&+\sum_{i_0\ne i_1,\ldots,i_k}(t_{i_0}^{2m-2}+t_{i_1}^{2m-2}+\cdots+t_{i_k}^{2m-2})t_{i_0}^2t_{i_1}^2\cdots t_{i_k}^2.
\end{align*}
Then since $\P^1p_n=2\sum_{1\le i_1<\cdots<i_n\le m}(t_{i_1}^{2q}+\cdots+t_{i_n}^{2q})t_{i_1}^2\cdots t_{i_n}^2$, we obtain the first equality. The second equality is obvious. 
\end{proof}

\begin{proposition}
\label{PE8}
Define ideals $I_k$ of $\Z/p[p_1,\ldots,p_7,c_8]$ for $k=0,\ldots,8$ as
\begin{align*}
I_0&=(p_1,p_2^2,p_3^3,p_4^2,p_6^2,c_8),&I_1&=I_0+(p_3,p_6),&I_2&=I_0+(p_2,p_3^2,p_4,p_7^2),\\
I_3&=I_0+(p_2,p_3^2,p_6),&I_4&=I_0+(p_2,p_3^2,p_4),&I_5&=I_0+(p_2,p_3,p_4,p_6,p_7),\\
I_6&=I_0+(p_2,p_3^2,p_4,p_6),&I_7&=I_0+(p_2,p_3^2,p_4,p_6,p_7^2),&I_8&=I_0+(p_2,p_4,p_7^4,\hat{x}_{24}).
\end{align*}
Then for a generator $x_k\in H^*(B\mathrm{E}_8;\Z/p)$, we have the following table.
\renewcommand{\arraystretch}{1.2}
\begin{table}[H]
\centering
\begin{tabular}{l|l|l|l||l|l|l|l}
$p$&$k$&$\P^1\rho_1^*(x_k)\mod I$&$I$&$p$&$k$&$\P^1\rho_1^*(x_k)\mod I$&$I$\\
\hline
$31$&$16$&$9p_7^2p_5+24p_7p_5^2p_2 + 22p_5^3p_4$&$I_1$&$37$&$4$&$p_7^2p_5+34p_7p_5^2p_2 + 36p_5^3p_4$&$I_1$\\
&$24$&$28 p_7 p_6 p_5 p_3 + 16 p_6p_5^3 $&$I_2$&&$16$&$8p_7^2p_5p_3+27p_7p_5^3+2p_5^3p_4p_3$&$I_3$\\
&$28$&$27p_7^2p_5p_3+30p_7p_5^3+30p_5^3p_4p_3$&$I_3$&&$24$&$5p_7^3p_3 + 27p_7^2p_5^2+36p_6p_5^3p_3$&$I_4$\\
&$36$&$p_7^3p_3 + 10p_7^2p_5^2+6p_6p_5^3p_3$&$I_4$&&$28$&$7p_5^5$&$I_5$\\
&$40$&$8p_5^5$&$I_5$&&$36$&$20p_7^2p_5^2p_3 + 35p_7p_5^4$&$I_6$\\
&$48$&$4p_7^2p_5^2p_3 + 5p_7p_5^4$&$I_6$&&$48$&$36p_7p_5^4p_3 + 3p_5^6$&$I_7$\\\hline
\end{tabular}
\end{table}

\renewcommand{\arraystretch}{1.2}
\begin{table}[H]
\centering
\begin{tabular}{l|l|l|l||l|l|l|l}
\hline
$41$&$4$&$35p_7p_6p_5p_3+40p_6p_5^3$&$I_2$&$43$&$4$&$3p_7^2p_5p_3+p_7p_5^3+39 p_5^3p_4p_3$&$I_3$\\
&$16$&$9p_7^3p_3 + 38p_7^2p_5^2+16p_6p_5^3p_3$&$I_4$&&$16$&$9 p_5^5$&$I_5$\\
&$28$&$7p_7^2p_5^2p_3 + 6p_7p_5^4$&$I_6$&&$24$&$11 p_7^2p_5^2p_3 + 40 p_7p_5^4$&$I_6$\\
&$40$&$34p_7p_5^4p_3 + 16p_5^6$&$I_7$&&$36$&$35 p_7p_5^4p_3 + 42 p_5^6$&$I_7$\\\hline
$47$&$4$&$p_7^3p_3 + 25 p_7^2p_5^2+43 p_6p_5^3p_3$&$I_4$&$53$&$4$&$6 p_7^2p_5^2p_3 + p_7p_5^4$&$I_6$\\
&$16$&$35 p_7^2p_5^2p_3 + 10 p_7p_5^4$&$I_6$&&$16$&$23 p_7p_5^4p_3 + 39 p_5^6$&$I_7$\\\cline{5-8}
&$28$&$17 p_7p_5^4p_3 + 23 p_5^6$&$I_7$&$59$&$4$&$5 p_7 p_5^4 p_3 + 10 p_5^6$&$I_7$\\\hline
\end{tabular}
\end{table}
\noindent  For $p=31$, we also have 
$$\P^1\rho_1^*(x_{48})\equiv 17p_7^3p_3^2+4p_7^2p_5^2p_3+5p_7p_5^4,\quad\P^2\rho_1^*(x_{48})\equiv 26 p_7^3p_5^3p_3^2+5p_7^2p_5^5p_3+8p_7p_5^7\mod I_8.$$
\end{proposition}

\begin{proof}
For $i=4,16,24,28,36$, we have $\rho_1^*(x_i)\equiv\hat{x}_i\mod(p_1^2)$. Since $\P^1(p_1^2)\subset(p_1)$ by the Cartan formula, we have $\P^1\rho_1^*(x_i)\equiv\P^1\hat{x}_i\mod(p_1)$. For $i=40,48$, we analogously have $\P^1\rho_1^*(x_i)=\P^1\hat{x}_i+(\P^1p_1)q$ for some polynomial $q$ in $p_2,\ldots,p_7,c_8$. For a degree reason, we have $q\equiv 0\mod(p_1,p_2,p_3^2,p_4,p_6,c_8)$, implying that $\P^1\rho_1^*(x_i)\equiv\P^1\hat{x}_i\mod I$ for the prescribed ideal $I$. Thus in order to fill the table, we only need to calculate $\P^1\hat{x}_i$ by Lemma \ref{P1} and Girard's formula \eqref{Girard}. 

For $p=31$, we have $\P^1\rho_1^*(x_{48})\equiv\P^1\hat{x}_{48}+(\P^1p_1)q\mod(p_1)$ for some polynomial $q$ in $p_2,\ldots,p_7,c_8$ as above. Since $\hat{x}_i\in I_8$ for $i=4,16,24,36$, we have $\P^1p_1\equiv 0\mod I_8$ for a degree reason, hence $\P_1\rho_1^*(x_{48})\equiv\P^1\hat{x}_{48}\mod I_8$. Then we can calculate $\P^1\rho_1^*(x_{48})\mod I_8$ by Lemma \ref{P1}. Since $\P^2p_1=p_1^p$ and $\rho_1^*(x_{48})\equiv\hat{x}_{48}\mod(p_1)$, we have $\P^2\rho_1^*(x_{48})\equiv\P^2\hat{x}_{48}\mod(p_1)$. Now $\P^2\rho_1(x_{48})$ for $p=31$ can be calculated from Lemma \ref{P1} and Girard's formula \eqref{Girard} together with the Adem relation $\P^1\P^1=2\P^2$.
\end{proof}

Quite similarly to Proposition \ref{PE8}, we can calculate $\P^1\rho_i^*(x_j)$ for $G=\mathrm{E}_7,\mathrm{E}_6$.

\begin{proposition}
\label{PE7}
For a generator $x_k\in H^*(B\mathrm{E}_7;\Z/p)$, we have the following table.
\renewcommand{\arraystretch}{1.2}
\begin{table}[H]
\centering
\begin{tabular}{l|l|l|l}
$p$&$k$&$\P^1\rho_2^*(x_k)\mod I$&$I$\\
\hline
$19$&$12$&$18p_5^2p_2+3p_5p_4p_3+15p_5p_3p_2^2+10p_4^3+17p_4^2p_2^2+6p_4p_2^4+15p_2^6$&$(p_1,p_3^2,c_6)$\\
&$16$&$11p_5p_4^2+16p_5p_4p_2^2+15p_5p_2^4$&$(p_1,p_3,c_6)$\\
&$20$&$p_5^2p_4+18p_5^2p_2^2+17p_5p_4p_3p_2+p_5p_3p_2^3+4c_6p_5p_4p_2+12c_6p_5p_2^3$&$(p_1,p_3^2,c_6^2)$\\
&&$+16c_6p_4^2p_3+8p_4c_6p_3p_2^2+7c_6p_3p_2^4$\\
&$24$&$13p_5p_4^2p_2+7p_5p_4p_2^3+8p_5p_2^5$&$(p_1,p_3,p_5^2,c_6)$\\
&$28$&$14p_5^2p_4p_2+p_5^2p_2^3+8p_5p_4^2p_3+10p_5p_4p_3p_2^2+17p_5p_3p_2^4+p_4^4+9p_4^3p_2^2$&$(p_1,p_3^2,c_6^2)$\\
&&$+6p_4^2p_2^4+p_4p_2^6+3p_2^8$\\
&$36$&$9p_5^2p_4^2+4p_5^2p_4p_2^2+6p_5^2p_2^4+17p_5p_4^2p_3p_2+15p_5p_3p_2^5+4p_4^4p_2+5p_4^3p_2^3$&$(p_1,p_3^2,c_6^2)$\\
&&$+2p_4^2p_2^5+11p_4p_2^7+3p_2^9$\\\hline
\end{tabular}
\end{table}

\renewcommand{\arraystretch}{1.2}
\begin{table}[H]
\centering
\begin{tabular}{l|l|l|l}
\hline
$23$&$4$&$22p_5^2p_2+21p_5p_4p_3+3p_5p_3p_2^2+15p_4^3+13p_4^2p_2^2+22p_4p_2^4+4p_2^6$&$(p_1,p_3^2,c_6)$\\
&$12$&$7p_5^2p_4+6p_5^2p_2^2+14p_5p_4p_3p_2+13p_5p_3p_2^3+10p_4^3p_2+18p_4^2p_2^3+21p_4p_2^5$&$(p_1,p_3^2,c_6^2)$\\
&&$+4p_2^7+14c_6p_5p_4p_2+16c_6p_5p_2^3+7c_6p_4^2p_3+2p_4c_6p_3p_2^2+7c_6p_3p_2^4$\\
&$16$&$3p_5p_4^2p_2+20p_5p_4p_2^3+19p_5p_2^5$&$(p_1,p_3,p_5^2,c_6)$\\
&$28$&$9p_5^2p_4^2+3p_5^2p_4p_2^2+2p_5^2p_2^4+10p_5p_4^2p_3p_2+10p_5p_4p_3p_2^3+8p_5p_3p_2^5$&$(p_1,p_3^2,c_6^2)$\\
&&$+14p_4^4p_2+15p_4^3p_2^3+14p_2^9+9p_4^2p_2^5+15p_4p_2^7$\\\hline
$29$&$4$&$26p_5p_4^2p_2+4p_5p_4p_2^3+28p_5p_2^5$&$(p_1,p_3,p_5^2,c_6)$\\
&$16$&$19p_5^2p_4^2+p_5^2p_4p_2^2+19p_5^2p_2^4+10p_5p_4^2p_3p_2+6p_5p_4p_3p_2^3+13p_5p_3p_2^5$&$(p_1,p_3^2,c_6^2)$\\
&&$+p_4^4p_2+7p_4^3p_2^3+2p_4^2p_2^5+16p_4p_2^7+21p_2^9$\\\hline
$31$&$12$&$p_5^3p_3+17p_5^2p_4^2+10p_5^2p_4p_2^2+28p_5^2p_2^4+4p_5p_4^2p_3p_2+18p_5p_4p_3p_2^3$&$(p_1,p_3^2,c_6^2)$\\
&&$+21p_2p_4^4+3p_4^3p_2^3+6p_4p_2^7+4p_5^3p_2^9+10c_6p_5^3+3c_6p_5^2p_3p_2+3c_6p_5p_4^2p_2$\\
&&$+27c_6p_5p_4p_2^3+c_6p_5p_2^5+c_6p_4^3p_3+25c_6p_4^2p_3p_2^2+5c_6p_4p_3p_2^4+30c_6p_3p_2^6$\\\hline
\end{tabular}
\end{table}
\end{proposition}

\begin{proposition}
\label{PE6}
For a generator $x_k\in H^*(B\mathrm{E}_6;\Z/p)$, we have the following table.
\renewcommand{\arraystretch}{1.2}
\begin{table}[H]
\centering
\begin{tabular}{l|l|l|l}
$p$&$k$&$\P^1\rho_3^*(x_k)\mod I$&$I$\\\hline
$13$&$10$&$6c_5p_4p_2+11c_5p_2^3$&$(p_1,p_3^2,c_5^2)$\\
&$12$&$10p_4p_3p_2+12p_3p_2^3+4c_5^2p_4+c_5^2p_2^2$&$(p_1,p_3^2)$\\
&$16$&$5p_2^5$&$(p_1,p_3,p_4,c_5)$\\
&$18$&$5c_5p_4^2+9c_5p_4p_2^2+7c_5p_2^4$&$(p_1,p_3,c_5^2)$\\
&$24$&$p_4^3+4p_4^2p_2^2+12p_4p_2^4+7p_2^6$&$(p_1,p_3,c_5)$\\\hline
$17$&$4$&$2p_4p_3p_2+16p_3p_2^3+16c_5^2p_4+c_5^2p_2^2$&$(p_1,p_3^2)$\\
&$10$&$4c_5p_4^2+9c_5p_4p_2^2+2c_5p_2^4$&$(p_1,p_3,c_5^2)$\\
&$16$&$11p_4^3+p_4^2p_2^2+8p_4p_2^4+8p_2^6$&$(p_1,p_3,c_5)$\\\hline
\end{tabular}
\end{table}
\end{proposition}

We finally calculate $\P^1x_k$ for a generator $x_k\in H^*(B\mathrm{G}_2;\Z/p)$.

\begin{proposition}
\label{PG2}
For a generator $x_k\in H^*(B\mathrm{G}_2;\Z/p)$, we have 
$$\P^1x_k=\begin{cases}x_4x_{12}+2x_4^4&(k,p)=(4,7)\\6x_{12}^2+2x_4^3x_{12}&(k,p)=(12,7)\\6x_{12}^2+x_4^3x_{12}+2x_4^6&(k,p)=(4,11).\end{cases}$$
\end{proposition}

\begin{proof}
By Proposition \ref{G2} and the naturality of $\P^1$, we have $\P^1x_{4k}=\P^1\rho^*(p_k)=\rho^*(\P^1p_k)$, hence the proof is completed by Lemma \ref{P1} and Girard's formula \eqref{Girard}.
\end{proof}

\section{Proof of Theorem \ref{P}}

In this section, we prove Theorem \ref{P} by summarizing results in the previous section.

\subsection{The case of $\mathrm{E}_8$}

Suppose that $\mathrm{E}_8$ is $p$-regular, that is, $p>30$. By an easy degree consideration, we see that if $\P^1x_k\mod (x_{2i}\,\vert\,i\in\t(\mathrm{E}_8))^3$ is nontrivial for a generator $x_k$ of $H^*(B\mathrm{E}_8;\Z/p)$, it is as in the following table.

\renewcommand{\arraystretch}{1.2}
\begin{table}[H]
\centering
\begin{tabular}{l|l|l}
&$\P^1x_k\mod (x_{2i}\,\vert\,i\in\t(\mathrm{E}_8))^3$&$(k,p)$\\
\hline
(1)&$\lambda_1x_4x_{60}+\lambda_2x_{16}x_{48}+\lambda_3x_{24}x_{40}+\lambda_4x_{28}x_{36}$&$(4,31)$\\
(2)&$\lambda_1x_{16}x_{60}+\lambda_2x_{28}x_{48}+\lambda_3x_{36}x_{40}$&$(16,31),(4,37)$\\
(3)&$\lambda_1x_{24}x_{60}+\lambda_2x_{36}x_{48}$&$(24,31),(4,41)$\\
(4)&$\lambda_1x_{28}x_{60}+\lambda_2x_{40}x_{48}$&$(28,31),(16,37),(4,43)$\\
(5)&$\lambda_1x_{36}x_{60}+\lambda_2x_{48}^2$&$(36,31),(24,37),(16,41),(4,47)$\\
(6)&$\lambda x_{40}x_{60}$&$(40,31),(28,37),(16,43)$\\
(7)&$\lambda x_{48}x_{60}$&$(48,31),(36,37),(28,41),(24,43),$\\
&&$(16,47),(4,53)$\\
(8)&$\lambda x_{60}^2$&$(60,31),(48,37),(40,41),(36,43),$\\
&&$(28,47),(16,53),(4,59)$\\\hline
\end{tabular}
\end{table}

Let $I_k$ for $k=1,\ldots,8$ be the ideals of $\Z/p[p_1,\ldots,p_7,c_8]$ as in Proposition \ref{PE8}.

\noindent(1) It is proved in \cite{HK} that $\lambda_i\ne 0$ for $i=1,2,3,4$.

\noindent(2) Since $\hat{x}_i\in I_1$ for $i=4,16,24$, for a degree reason, we have 
$$\rho_1^*(\P^1x_k)\equiv\lambda_2\hat{x}_{28}\hat{x}_{48}+\lambda_3\hat{x}_{36}\hat{x}_{40}\equiv24000(\lambda_2(4p_7^2p_5+p_7p_5^2p_2)+\lambda_3p_7p_5^2p_2)\mod I_1+(p_4).$$
On the other hand, by the naturality of $\P^1$ and Proposition \ref{PE8}, 
$$\rho_1^*(\P^1x_k)=\P^1\rho_1(x_k)\equiv\begin{cases}24p_7p_5^2p_2 + 9p_7^2p_5&(p=31)\\34p_7p_5^2p_2 + p_7^2p_5&(p=37)\end{cases}\mod I_1+(p_4),$$
implying that $(\lambda_2,\lambda_3)=(10,14),(-9,6)$ according as $p=31,37$. Since $\hat{x}_4,\hat{x}_{16}^2,\hat{x}_{24},\hat{x}_{36}\in I_1+(p_2,p_7)$, we also have
$$\rho_1^*(\P^1x_k)\equiv\lambda_1\hat{x}_{16}\rho_1^*(x_{60})+\lambda_3\hat{x}_{36}\hat{x}_{40}\equiv\lambda_1\hat{x}_{16}\rho_1^*(x_{60})-1500\lambda_3p_5^3p_4\mod I_1+(p_2,p_7),$$
and by Proposition \ref{PE8}, 
$$\rho_1^*(\P^1x_k)=\P^1\rho_1^*(x_k)\equiv\begin{cases}22p_5^3p_4&(p=31)\\36p_5^3p_4&(p=37)\end{cases}\mod I_1+(p_2,p_7),$$
implying that $\lambda_1\ne 0$ for both $p=31,37$.

\noindent(3) Since $\hat{x}_i,\hat{x}_j^2\in I_2$ for $i=4,16$ and $j=24,28,36$, we have
$$\rho_1^*(\P^1x_k)\equiv\lambda_1\hat{x}_{24}\rho_1^*(x_{60})+\lambda_2\hat{x}_{36}\hat{x}_{48}\equiv\lambda_1\hat{x}_{24}\rho_1^*(x_{60})-14400\lambda_2p_7p_6p_5p_3\mod I_2.$$
By the naturality of $\P^1$ and Proposition \ref{PE8}, we also have
$$\rho_1^*(\P^1x_k)=\P^1\rho_1^*(x_k)\equiv\begin{cases}28p_7p_6p_5p_3+16p_6p_5^3&(p=31)\\35p_7p_6p_5p_3+40p_6p_5^3&(p=41)\end{cases}\mod I_2,$$
implying that $\lambda_1\ne 0$ and $\lambda_2\ne 0$ for both $p=31,41$.

\noindent(4) Since $\hat{x}_i,\hat{x}_{28}^2\in I_3+(p_3,p_4,p_7^2,\hat{x}_{40})$ for $i=4,16,24,36,40$, it follows from Proposition \ref{PE8} that
$$\lambda_1\hat{x}_{28}\rho_1^*(x_{60})\equiv\rho_1^*(\P^1x_k)\equiv\P^1\rho_1^*(x_k)\not\equiv 0\mod I_3+(p_3,p_4,p_7^2,\hat{x}_{40})$$
so $\lambda_1\ne 0$. We can similarly get $\lambda_2\ne 0$ by considering $\rho_1^*(\P^1x_k)\mod I_3+(p_7^2,\hat{x}_{28})$ since $\hat{x}_i\in I_3+(p_7^2,\hat{x}_{28})$ for $i=4,16,24,28$. 

\noindent(5), (6) and (7) We get $\lambda\ne 0$ similarly to (4) by considering $\rho_1^*(\P^1x_k)$ modulo the ideals $I_4+(p_7),I_5,I_6+(\hat{x}_{40}^2)$ respectively for (5), (6) and (7) since $\hat{x}_4,\hat{x}_{16},\hat{x}_{24}^2,\hat{x}_{36}^2\in I_4+(p_7)$, $\hat{x}_i\in I_5$ for $i=4,16,24,18,36$ and $\hat{x}_i\in I_6+(\hat{x}_{40}^2)$ for $i=4,16,24,36,40$.




\noindent(8) Suppose $(k,p)\ne(60,31)$. Since $\hat{x}_i,\hat{x}_{28}^2,\hat{x}_{40}^3\in I_7+(\hat{x}_{40}^3)$ for $i=4,16,24,36$, we get $\lambda\ne 0$ by considering $\rho_1^*(\P^1x_k)\mod I_7+(\hat{x}_{40}^3)$ as above. 

Suppose next that $(k,p)=(60,31)$. By a degree reason, we have 
$$\rho_1^*(x_{60})\equiv\alpha p_5^3+\beta p_7p_5p_3\mod I_8+(\hat{x}_{40}^2)$$ 
for $\alpha,\beta\in\Z/p$. Since $\hat{x}_i,\hat{x}_{40}^2\in I_8+(\hat{x}_{40}^2)$ for $i=4,16,24,36$ and $\rho_1^*(x_{48})\equiv-200p_7p_5\mod I_8$, we have
$$\rho_1^*(\P^1x_{48})\equiv\mu\hat{x}_{48}\rho_1^*(x_{60})\equiv-200\mu(\alpha p_7p_5^4+\beta p_7^2p_5^2p_3)\mod I_8+(\hat{x}_{40}^2)$$
for some $\mu\in\Z/p$. By Proposition \ref{PE8}, we also have
$$\rho_1^*(\P^1x_{48})=\P^1\rho_1^*(x_{48})\equiv10p_7p_5^4+11p_7^2p_5^2p_3\mod I_8+(\hat{x}_{40}^2)$$
Then we may put $(\alpha,\beta)=(17,28)$ and $\mu=1$. In the case (7), we have seen that $\P^1x_{48}\equiv\mu x_{48}x_{60}\mod (x_{2i}\,\vert\,i\in\t(\mathrm{E}_8))^3$, implying that $\P^1\P^1x_{48}\equiv(\lambda+1)x_{48}x_{60}^2\mod (x_{2i}\,\vert\,i\in\t(\mathrm{E}_8))^4$, where $\P^1x_{60}\equiv\lambda x_{60}^2\mod (x_{2i}\,\vert\,i\in\t(\mathrm{E}_8))^3$. Then for a degree reason, we get
$$\rho_1^*(\P^1\P^1x_{48})\equiv(\lambda+1)\hat{x}_{48}\rho_1^*(x_{60})^2\equiv 21(\lambda+1)p_7^3p_5^3p_3^2\mod I_8+(\hat{x}_{40}^2).$$
On the other hand, by the Adem relation $\P^1\P^1=2\P^2$ and Proposition \ref{PE8}, we have
$$\rho_1^*(\P^1\P^1x_{48})=\rho_1^*(2\P^2x_{48})=2\P^2\rho_1^*(x_{48})\equiv7p_7^3p_5^3p_3^2\mod I_8+(\hat{x}_{40}^2),$$
hence $\lambda\ne 0$.

\subsection{The case of $\mathrm{E}_7$}

Suppose that $\mathrm{E}_7$ is $p$-regular, that is, $p>18$. Then if $\P^1x_k\mod (x_{2i}\,\vert\,i\in\t(\mathrm{E}_7))^3$ is non-trivial, it is as in the following table.

\renewcommand{\arraystretch}{1.2}
\begin{table}[H]
\centering
\begin{tabular}{l|l|l}
&$\P^1x_k\mod (x_{2i}\,\vert\,i\in\t(\mathrm{E}_7))^3$&$(k,p)$\\
\hline
(1)&$\lambda_1x_4x_{36}+\lambda_2x_{12}x_{28}+\lambda_3x_{16}x_{24}+\lambda_4x_{20}^2$&$(4,19)$\\
(2)&$\lambda_1x_{12}x_{36}+\lambda_2x_{20}x_{28}+\lambda_3x_{24}^2$&$(12,19),(4,23)$\\
(3)&$\lambda_1x_{16}x_{36}+\lambda_2x_{24}x_{28}$&$(16,19)$\\
(4)&$\lambda_1x_{20}x_{36}+\lambda_2x_{28}^2$&$(20,19),(12,23)$\\
(5)&$\lambda x_{24}x_{36}$&$(24,19),(16,23),(4,29)$\\
(6)&$\lambda x_{28}x_{36}$&$(28,19),(4,31)$\\
(7)&$\lambda x_{36}^2$&$(36,19),(28,23),(16,29),(12,31)$\\\hline
\end{tabular}
\end{table}

\noindent(1) It is proved in \cite{HK} that $\lambda_i\ne 0$ for $i=1,2,3,4$.

\noindent(2) Put $I=(p_1,p_3^2,c_6,\hat{x}_{16})$. Since $\hat{x}_4,\hat{x}_{12}^2,\hat{x}_{16}\in I$, by Corollary \ref{E7}, we have
$$\rho_2^*(\P^1x_k)\equiv\lambda_1\hat{x}_{12}\hat{x}_{36}+\lambda_2\hat{x}_{20}\hat{x}_{28}+\lambda_3\hat{x}_{24}^2\equiv 60\lambda_1p_5p_3p_2^2+40\lambda_2p_5^2p_2+\tfrac{25}{81}\lambda_3p_2^6\mod I.$$
On the other hand, it follows from Proposition \ref{PE7} that
$$\rho_2^*(\P^1x_k)=\P^1\rho_2(x_k)\equiv\begin{cases}18p_5^2p_2+10p_5p_3p_2^2+p_2^6&(p=19)\\22p_5^2p_2+7p_5p_3p_2^2+7p_2^6&(p=23)\end{cases}\mod I,$$
hence $\lambda_1\ne 0$, $\lambda_2\ne 0$ and $\lambda_3\ne 0$.

\noindent(3)  In this case, we have $(k,p)=(16,19)$. Put $I=(p_1,p_3,c_6,\hat{x}_{16}^2)$. Since $\hat{x}_4,\hat{x}_{12},\hat{x}_{16}^2\in I$, it follows from Proposition \ref{E7} that
$$\rho_2^*(\P^1x_{16})\equiv\lambda_1\hat{x}_{16}\hat{x}_{36}+\lambda_2\hat{x}_{24}\hat{x}_{28}
\equiv(13\lambda_1+9\lambda_2)p_5p_4p_2^2+(9\lambda_1+14\lambda_2)p_5p_2^4\mod I.$$
By Proposition \ref{PE7}, we also have $\rho_2^*(\P^1x_{16})=\P^1\rho_2^*(x_{16})\equiv11p_5p_4p_2^2+14p_5p_2^4\mod I$ , implying $\lambda_1\ne 0$ and $\lambda_2\ne 0$.

\noindent(4) Put $I=(p_1,p_3^2,c_6^2,\hat{x}_{12},\hat{x}_{16}^2,\hat{x}_{24},\hat{x}_{16}\hat{x}_{20}^2)$. Since $\hat{x}_i,\hat{x}_{16}^2,\hat{x}_{16}\hat{x}_{20}^2\hat{x}_{24}\in I$ for $i=4,12,24$, we have
$$\rho_2^*(\P^1x_k)\equiv\lambda_1\hat{x}_{20}\hat{x}_{36}+\lambda_2^2\hat{x}_{28}^2\equiv(-10\lambda_1+1600\lambda_2)p_5^2p_2^2+(\tfrac{2}{3}\lambda_1-\tfrac{320}{3}\lambda_2)p_5p_3p_2^3\mod I.$$
By Proposition \ref{PE7}, we also have
$$\rho_2^*(\P^1x_k)=\P^1\rho_2^*(x_k)\equiv\begin{cases}
10p_5^2p_2^2+12p_5p_3p_2^3&(p=19)\\
15p_5^2p_2^2+22p_5p_3p_2^3&(p=23)\end{cases}\mod I,$$
hence $\lambda_1\ne 0$ and $\lambda_2\ne 0$.

\noindent(5) and (7) Put $I=(p_1,p_3,p_5^2,c_6,\hat{x}_{16})$ and $J=(p_1,p_3^2,c_6^2,\hat{x}_{12},\hat{x}_{16},\hat{x}_{20}^3,\hat{x}_{24}^2,\hat{x}_{20}\hat{x}_{24}\hat{x}_{28})$. Then since  $\hat{x}_i,\hat{x}_{20}^2\in I$ for $i=4,12,16$ and $\hat{x}_i,\hat{x}_{20}^3,\hat{x}_{24}^2,\hat{x}_{20}\hat{x}_{24}\hat{x}_{28}\in J$ for $i=4,12,16$, we have $\lambda\ne 0$ similarly to (4) of $\mathrm{E}_8$ by considering $\rho_2^*(\P^1x_k)$ modulo $I$ and $J$ respectively for (5) and (7).


\noindent(6) The case $p=31$ follows from the above case of $\mathrm{E}_8$ together with Corollary \ref{E7}. Then we consider the case $p=19$. Put $I=(p_1,p_3^2,c_6^2,\hat{x}_{12},\hat{x}_{16},\hat{x}_{20}^2,\hat{x}_{24}^2)$. Since $\hat{x}_i,\hat{x}_j^2\in I$ for $i=4,12,16$ and $j=20,24$, we get $\lambda\ne 0$ as above by considering $\rho_2^*(\P^1x_k)\mod I$.


\subsection{The cases of $\mathrm{E}_6$ and $\mathrm{F}_4$}

We first consider the case of $\mathrm{E}_6$. Suppose that $\mathrm{E}_6$ is $p$-regular, that is, $p\ge 13$. By an easy dimensional consideration, we see that if $\P^1x_k\not\equiv 0\mod (x_{2i}\,\vert\,i\in\t(\mathrm{E}_6))^3$, it is as in the following table.
\renewcommand{\arraystretch}{1.2}
\begin{table}[H]
\centering
\begin{tabular}{l|l|l}
&$\P^1x_k\mod (x_{2i}\,\vert\,i\in\t(\mathrm{E}_6))^3$&$(k,p)$\\\hline
(1)&$\lambda_1x_4x_{24}+\lambda_2x_{10}x_{18}+\lambda_3x_{12}x_{16}$&$(4,13)$\\
(2)&$\lambda_1x_{10}x_{24}+\lambda_2x_{16}x_{18}$&$(10,13)$\\
(3)&$\lambda_1x_{12}x_{24}+\lambda_2x_{18}^2$&$(12,13),(4,17)$\\
(4)&$\lambda x_{16}x_{24}$&$(16,13),(4,19)$\\
(5)&$\lambda x_{18}x_{24}$&$(18,13),(10,17)$\\
(6)&$\lambda x_{24}^2$&$(24,13),(16,17),(12,19),(4,23)$\\\hline
\end{tabular}
\end{table}
When $p=19,23$, the result follows from the above case of $\mathrm{E}_7$ and Corollary \ref{E6}.

\noindent(1) It is proved in \cite{HK} that $\lambda_1\ne 0$, $\lambda_2\ne 0$ and $\lambda_3\ne 0$.

\noindent(2) Put $I=(p_1,p_3^2,c_5^2)$. Since $\hat{x}_4,\hat{x}_{10}^2,\hat{x}_{12}^2\in I$, we have
$$\rho_3^*(\P^1x_{10})\equiv\lambda_1\hat{x}_{10}\hat{x}_{24}+\lambda_2\hat{x}_{16}\hat{x}_{18}\equiv 5\lambda_1(-p_4p_2c_5+\tfrac{1}{36}p_2^3c_5)+\lambda_2(12p_4p_2c_5+p_2^3c_5)\mod I,$$
where $\hat{x}_{10}=c_5$ and $\hat{x}_{18}=p_2c_5$. On the other hand, by Proposition \ref{PE6}, we have $\rho_3^*(\P^1x_{10})=\P^1\rho_3^*(x_{10})\equiv 6p_4p_2c_5+7p_2^3c_5\mod I$ for $p=13$, hence $\lambda_1\ne 0$ and $\lambda_2\ne 0$.

\noindent(3) Put $I=(p_1,p_3^2,\hat{x}_{16})$. It is sufficient to consider the case $p=13,17$. Since $\hat{x}_i,\hat{x}_{12}^2\in I$ for $i=4,16$, 
$$\rho_3^*(\P^1x_k)\equiv\lambda_1\hat{x}_{12}\hat{x}_{24}+\lambda_2\hat{x}_{18}^2\equiv-\tfrac{10}{3}\lambda_1p_3p_2^3+\lambda_2p_2^2c_5^2\mod I.$$
By Proposition \ref{PE6}, we have
$$\rho_3^*(\P^1x_k)=\P^1\rho_3^*(x_k)\equiv\begin{cases}9p_3p_2^3+5c_5^2p_2^2&(p=13)\\13p_3p_2^3-11c_5^2p_2^2&(p=17)\end{cases}\mod I,$$
implying $\lambda_1\ne 0$ and $\lambda_2\ne 0$.

\noindent(4), (5) and (6) Put $I=(p_1,p_3,p_4,c_5)$, $J=(p_1,p_3,c_5^2,\hat{x}_{16})$ and $K=(p_1,p_3,c_5,\hat{x}_{16})$. Then since $\hat{x}_i\in I$ for $i=4,10,12$, $\hat{x}_i,\hat{x}_{10}^2\in J$ for $i=4,12,16$ and $\hat{x}_i\in K$ for $i=4,12,10,16$, we get $\lambda\ne 0$ similarly to (4) of $\mathrm{E}_8$ by considering $\rho_3^*(\P^1x_k)$ modulo $I,J,K$ respectively for (4), (5) and (6).




We next consider the case of $\mathrm{F}_4$. Notice that $\mathrm{F}_4$ is $p$-regular if and only if so is $\mathrm{E}_6$, and that as in the proof of Corollary \ref{F4}, the map $\alpha_3^*:H^*(B\mathrm{E}_6;\Z/p)\to H^*(B\mathrm{F}_4;\Z/p)$ is surjective. Then the result for $\mathrm{F}_4$ follows from that for $\mathrm{E}_6$ above.

\subsection{The case of $\mathrm{G}_2$}

For a degree reason, if $\mathrm{G}_2$ is $p$-regular and $\P^1x_k\not\equiv 0\mod(x_{2i}\,\vert\,i\in\t(\mathrm{G}_2))^3$, then $(k,p)=(4,7),(12,7),(4,11)$. Hence Theorem \ref{P} for $\mathrm{G}_2$ readily follows from Proposition \ref{PG2}.

\end{document}